\documentclass[12pt]{article}

\usepackage{amssymb}%
\usepackage{amsmath}%
\usepackage{amsthm}%
\usepackage{hyperref}%
\usepackage{graphicx}%
\usepackage{xcolor}%
\usepackage{mathrsfs}

\allowdisplaybreaks%

{\theoremstyle{plain}%
 \newtheorem{theorem}{Theorem}
 
 \newtheorem{lemma}{Lemma}%
}
{\theoremstyle{remark}
\newtheorem{remark}{Remark}
}
{\theoremstyle{definition}

}

\begin{document}

\begin{center}
{\large On the binary digits of the Erd\H{o}s--Borwein constant}

 \ 
 
{\sc John M. Campbell}

\vspace{0.1in}

{\footnotesize Department of Mathematics and Statistics}

{\footnotesize Dalhousie University}

{\footnotesize Halifax, NS B3H 4R2}

{\footnotesize Canada}

\vspace{0.1in}

{\footnotesize {\tt jh241966@dal.ca}}

\vspace{0.1in}

\end{center}

\begin{abstract}
 In a landmark paper on arithmetical properties of Lambert series, Erd\H{o}s proved that $ \sum_{n=1}^{\infty} \frac{1}{2^{n} - 1}$ is 
 irrational. This value $E$ is now referred to as the \emph{Erd\H{o}s--Borwein constant}. Crandall, in 2012, studied properties of the 
 base-2 expansion of this constant, and left the following as an open problem: Does the string $11$ occur infinitely often in the base-2 
 expansion of $E$? This open problem was also subsequently noted by Shallit. We succeed in introducing a full proof that solves 
 Crandall's problem in the affirmative. Our proof combines a congruence construction in the spirit of Erd\H{o}s and an estimate due to 
 Alford, Granville, and Pomerance for the counting function for primes in arithmetic progressions. Our argument was developed through 
 extensive interactions with GPT-5.5 Pro. 
\end{abstract}

\vspace{0.1in}

\noindent {\footnotesize \emph{MSC:} 11A63, 11A25}

\vspace{0.1in}

\noindent {\footnotesize \emph{Keywords:} divisor, digit, binary expansion, 
 Erd\H{o}s--Borwein constant, Lambert series, Chinese remainder theorem}

\section{Introduction}
 The arithmetic function given by the integer sequence $(d(n) : n \in \mathbb{N})$ for the number $d(n)$ of positive divisors of $n \in 
 \mathbb{N}$ often arises in deep and active areas of number theory. There are many open problems concerning the behavior of the given 
 sequence, and research devoted to these open problems helps to give light to broader areas related to divisibility properties of integers. 
 This motivates the purpose of our paper, which concerns an open problem due to Crandall \cite{Crandall2012} 
 on the \emph{Erd\H{o}s--Borwein constant} defined by 
\begin{equation}\label{displayE}
 E = \sum_{n=1}^{\infty} \frac{1}{2^{n} - 1} = \sum_{n=1}^{\infty} \frac{d(n)}{2^n}. 
\end{equation}
 Observe that the equivalence of the two series in \eqref{displayE} is justified by writing $$ \sum_{a=1}^{\infty} \frac{1}{2^{a} - 1} = 
 \sum_{a=1}^{\infty} \sum_{b=1}^{\infty} 2^{-ab} $$
 and by making use of absolute convergence. 

 \emph{Lambert series} broadly refer to (convergent) series of the form 
\begin{equation}\label{displayLaurent}
 \sum_{n=1}^{\infty} a_n \frac{q^n}{1-q^n} = \sum_{n=1}^{\infty} a_{n} \sum_{k=1}^{\infty} q^{n k} 
\end{equation}
 for a given sequence $(a_n : n \in \mathbb{N})$ and naturally arise in both number theory and the application of special functions, which 
 suggests an inherently interdisciplinary nature of infinite sums as in \eqref{displayLaurent}. Erd\H{o}s, in a 1948 paper 
 \cite{Erdos1948}, considered the function 
\begin{equation}\label{Erdosf}
 f(x) = \sum_{n=1}^{\infty} \frac{x^n}{1-x^n}
\end{equation}
 and proved that $f\left( \frac{1}{t} \right)$ is irrational for any integer $t$ such that $t > 1$ (cf.\ \cite{Vandehey2013}), with the $t = 2$ 
 case yielding the constant in \eqref{displayE}. Crandall, in a 2012 paper~\cite{Crandall2012}, studied the behavior of the binary digits of
 $E$, and applied bounds on the divisor function to determine the parity of
 $ \big\lfloor 2^{{10}^{100}} E \big\rfloor$. 
 In the same paper, 
 Crandall left it as an open problem to determine whether or not 
 the string $11$ appears infinitely often in the base-$2$ expansion 
\begin{equation}\label{EBnumerical}
 E = 1.1001101101010000010111111001111001000011111100100010\ldots, 
\end{equation}
 and this open problem was also highlighted by Shallit in the review of Crandall's work in the Mathematical Reviews database\footnote{See 
 the MathSciNet entry indexed as MR2988549.} In our current paper, we solve Crandall's problem in the affirmative, 
 via an argument developed through extensive interactions with GPT-5.5 Pro (see Acknowledgments section below). 
 Based on extant 
 research related to Crandall's work on the expansion in 
 \eqref{EBnumerical} \cite{AragonArtachoBaileyBorweinBorwein2013,Murakami2025,MurakamiTachiya2025}, 
 it appears that this problem has remained open, subsequent to the above 2012 paper from Crandall. 

 As suggested by Crandall \cite{Crandall2012}, the open problem given above is motivated by the problem of determining whether or 
 not $E$ is 2-normal, with regard to the work of Bailey and Crandall \cite{BaileyCrandall2002}, who proved necessary and sufficient 
 conditions for $E$ to be $2$-normal. The extent of research interest in Crandall's problem may also be seen in relation to a number of 
 notable research works concerning the Erd\H{o}s--Borwein constant. In this direction, Vandehey \cite{Vandehey2013} built upon 
 Erd\"os's proof \cite{Erdos1948} to show that $f\left( \frac{1}{b} \right)$ is irrational in base $|b|$ for integers $b < -1$. 

\section{Erd\H{o}s's irrationality proof}\label{secErdos}
 The key to Erd\H{o}s's proof \cite{Erdos1948} of the irrationality of \eqref{Erdosf} for the $x = \frac{1}{t}$ case for an integer $t > 1$ is 
 given by showing how the base-$t$ expansion of $f\big( \frac{1}{t} \big)$ contains strings of $0$-digits of arbitrary length, compared to 
 how this base-$t$ expansion is not finite. This is outlined below. 

 Let $t$ and $n$ be positive integers, and let $k = \big\lfloor \log^{1/10} n \big\rfloor$. Let $ p_{m_1}$, $p_{m_{2}}$, $\ldots$, 
 denote the consecutive 
 primes greater than $\log^2 n$, and set 
 $$ A = \prod_{1 \leq i \leq \frac{k(k+1)}{2}} p_{m_{i}}^{t}. $$
 Basic properties concerning prime distributions then give us that 
 $p_{m_i} < 2 \log^2 n$ for $ i \leq \frac{k(k+1)}{2}$. In turn, this can be used to show 
 that 
 $$ A < \left( 2 \log^2 n \right)^{t k^2} < e^{\log^{1/4} n}. $$
 Write $u = \frac{k(k-1)}{2} + 1$. 
 Erd\H{o}s's proof relies on solving the system 
\begin{align*}
 x & \equiv p_{m_{1}}^{t-1} \, \left(\operatorname{mod} \, p_{m_{1}}^{t} \right), \\ 
 x + 1 & \equiv \left( p_{m_{2}} p_{m_{3}} \right)^{t-1} \, \left(\operatorname{mod} \, \left( p_{m_{2}} p_{m_{3}} \right)^{t} 
 \right), \\ 
 & \cdots \\ 
 x + k - 1 & \equiv \left( p_{m_{u}} p_{m_{u+1}} \cdots p_{m_{u + k -1}} \right)^{t-1}
 \, \left(\operatorname{mod} \, \left( p_{m_u} p_{m_{u+1}} 
 \cdots p_{m_{u+k-1}} \right)^{t} \right) 
\end{align*}
 of congruences. 	 Integers satisfying the above simultaneous congruences are determined by the Chinese remainder theorem, and such 
 integers less than $n$ are of the form $$ x + y A $$ for $x$ and $y$ such that $0 < x < A$ and $0 \leq y < \big\lfloor \frac{n}{A} 
 \big\rfloor$. Erd\H{o}s's construction then gives us that 
 $$ d\left( x + y A + j \right) \equiv 0 \, \left(\operatorname{mod} \, t^{j+1} \right) $$
 for all $j$ such that $0 \leq j < k$. So, by expanding 
\begin{equation}\label{expandt}
 \sum_{r < x + y A + k} \frac{d(r)}{t^r} 
\end{equation}
 in base $t$, we obtain that $t^{-x-yA +1}$ will then be the lowest
 power of $t$ occurring. 

 Omitting details, by then arguing that the ``tail'' corresponding to \eqref{expandt}, i.e., the value
\begin{equation}\label{appropriatetail}
 \sum_{r \geq x + y A + k} \frac{d(r)}{t^r}, 
\end{equation}
 is appropriately bounded, one can conclude that the base-$t$ expansion of $f\big( \frac{1}{t} \big)$ will involve at least $\frac{k}{2}$ 
 consecutive zeroes. The remainder of Erd\H{o}s's argument is thus devoted to appropriately bounding the tail in 
 \eqref{appropriatetail}, and this largely relies on a variety of uses of bounds involving the divisor function or its partial sums, 
 referring again to Erd\H{o}s's original proof for details \cite{Erdos1948}. 

 \section{A full solution}
 A key to our solution to Crandall's problem relies on an application of the Chinese remainder theorem related to Erd\H{o}s's application 
 of the CRT \cite{Erdos1948} outlined in Section \ref{secErdos}. 
 We provide, as below, a formal statement of the CRT 
 appropriate for our purposes. 

 \ 

\noindent {\bf Chinese remainder theorem:} Let $M_{1}$, $M_{2}$, $\ldots$, $M_t$ be positive integers such that $(M_i, M_j) = 1$ for $i 
 \neq j$. Let $a_1$, $a_2$, $\ldots$, $a_t$ be arbitrary integers. Then the system of congruences of the form 
 $$ x \equiv a_i \pmod {{M}_{i}} $$
 for $1 \leq i \leq t$ has a solution $x \in \mathbb{Z}$. Moreover, 
 the solution is unique modulo $M_1 M_2 \cdots M_t$. 

 \ 

 As in the work of Vandehey related to the irrationality of the Erd\H{o}s--Borwein constant \cite{Vandehey2013}, we require the following 
 lemma attributed to Alford, Granville, and Pomerance \cite{AlfordGranvillePomerance1994}. As below, we write $\pi(N; d, a)$ in place 
 of the number of primes not exceeding $N$ and congruent to $a$ modulo $d$. This 
 counting function $\pi(N; d, a)$ 
 may be distinguished from the prime-counting function $\pi(x)$ giving the number of primes not 
 exceeding $x$. Also, as below, we are letting $\varphi$ denote Euler's totient 
 function, which is defined so that $\varphi(n)$ is equal to the number of integers $k \in [1, n]$ such that $(n, k) = 1$. 
 
 While the following result is properly attributed to 
 Alford et al., this is actually a consequence or a special case of
 material in their work given by Vandehey (see Remark \ref{remarkAGP} below). 
 
\begin{lemma}\label{Alfordetal}
 (Alford et al., 1994) For $\delta$ such that $0 < \delta < \frac{5}{12}$, there are positive integers $N_{0} = N_{0}(\delta)$ and 
 $\overline{\mathcal{D}} = \overline{\mathcal{D}}(\delta)$ such that 
\begin{equation}\label{countarithmetic} 
 \pi(N; d, a) \geq \frac{N}{ 2 \varphi(d) \log N } 
\end{equation}
 holds: 

\vspace{0.1in}

\noindent $\bullet$ For all $N > N_ 0$; 

\vspace{0.1in}

\noindent $\bullet$ For all moduli $d$ such that $1 \leq d \leq N^{\delta}$, apart from the possibility of those $d$ that are multiples of 
 certain elements in $\mathcal{D}(N)$, a set consisting of at most $\overline{\mathcal{D}}$ distinct integers all exceeding $\log N$; and 

\vspace{0.1in}

\noindent $\bullet$ For all integers $a$ coprime with $d$ \cite{AlfordGranvillePomerance1994} (cf.\ \cite{Vandehey2013}). 

\end{lemma}

\begin{remark}\label{remarkAGP}
 The above result, as given by Vandehey \cite{Vandehey2013}
 and attributed to Alford et al.\ \cite{AlfordGranvillePomerance1994}, 
 follows from a property of $\pi(y;d,a)$ given on page 705 of the given 
 Alford--Granville--Pomerance paper. 
\end{remark}

 We proceed to construct an average divisor bound associated with arithmetic sequences satisfying certain properties, as below. 

\begin{lemma}\label{basiclogbound}
 Let $a, A, M \in \mathbb{N}$, let $Y \geq 3$, and suppose that 
\begin{equation}\label{aA1} 
 (a, A) = 1 
\end{equation}
 and that 
\begin{equation}\label{aM1Ay} 
 a + (M - 1) A \leq Y. 
\end{equation}
 Then $$ \sum_{m=0}^{M-1} d(a + m A) \leq 2 M \left( 1 + \frac{1}{2} \log Y \right) + 2 \sqrt{Y}. $$ Moreover, if $ \sqrt{Y} \leq M \log 
 Y$, then 
\begin{equation}\label{withsqrtMlog}
 \sum_{m=0}^{M-1} d(a + m A) \leq 5 M \log Y. 
\end{equation}
\end{lemma}

\begin{proof}
 For $0 \leq m < M$, we write
\begin{equation}\label{defineNm} 
 N_{m} = a + m A, 
\end{equation}
 so that \eqref{aM1Ay} and \eqref{defineNm} together with the upper bound on $m$ give us that $N_{m} \leq Y$. For each divisor $u 
 \mid N$, pair $u$ with $\frac{N}{u}$. Since $N \leq Y$, we have $\sqrt{N} \leq \sqrt{Y}$. 
 For each pairing of an expression $ u$ with $ \frac{N}{u}$, 
 at least one expression among $u$ and $\frac{N}{u}$ is $\leq \sqrt{N} \leq \sqrt{Y}$. 
 So, every divisor of $N$ is accounted for by at most doubling the divisors $h \mid N$ satisfying $h \leq \sqrt{Y}$, 
 i.e., so that 
\begin{equation}\label{pairingdivisor}
 d(N) \leq 2 \sum_{\substack{h \leq \sqrt{Y} \\ h \mid N}} 1 
\end{equation}
 for $N \leq Y$. From \eqref{pairingdivisor}, we write
\begin{equation}\label{reducedouble}
 \sum_{m=0}^{M-1} d(a + m A) \leq 2 \sum_{m=0}^{M-1} \sum_{\substack{ h \leq \sqrt{Y} \\ h \mid (a + m A) }} 1, 
\end{equation}
 and we rewrite the upper bound in \eqref{reducedouble} so that 
\begin{equation}\label{rewritehash} 
 \sum_{m=0}^{M-1} d(a + mA) \leq 2 \sum_{h \leq \sqrt{Y}} \#\{ m \in [0, M) : 
 h \mid (a + m A) \}. 
\end{equation}
 Now, let $h \leq \sqrt{Y}$ be fixed, 
 and consider values $m \in [0, M)$ such that the congruence
\begin{equation}\label{amAequiv0}
 a + m A \equiv 0 \pmod h 
\end{equation} 
 holds. We proceed to consider two cases, as below. 

 First, suppose that $(h, A) > 1$, writing $g = (h, A) > 1$. So, if $h \mid (a + mA)$ (noting that this is equivalent to 
 \eqref{amAequiv0}), then 
\begin{equation}\label{gmidamA}
 g \mid (a + m A). 
\end{equation}
 However, since $g \mid A$, we have that $g \mid mA$. This consequence together with \eqref{gmidamA} allow us to deduce that $g 
 \mid a$, but $g > 1$ and we thus have a common divisor of $a$ and $A$ strictly greater than $1$, 
 contradicting \eqref{aA1}. 
 So, we have shown that: If $(h, A) > 1$, then no value $m \in [0, M)$ satisfies the congruence in \eqref{amAequiv0}. 

 Now, suppose that $(h, A) = 1$. If $h = 1$, then every integer $m \in [0, M)$ satisfies the congruence in \eqref{amAequiv0}, so that the 
 number of solutions $m \in [0, M)$ to \eqref{amAequiv0}
 is $M \leq \frac{M}{h} + 1$.
 If $h > 1$, then $A$ is invertible (multiplicatively) modulo $h$, 
 so that the relation in \eqref{amAequiv0} is equivalent to
\begin{equation}\label{minresidue} 
 m \equiv -aA^{-1} \pmod h. 
\end{equation}
 In general, we have that the number of integers in $ [0, M)$ and in a fixed residue class modulo $h$ is bounded above by $ \big\lceil 
 \frac{M}{h} \big\rceil \leq \frac{M}{h} + 1$. So, since $m$ is necessarily in the residue class in \eqref{minresidue}, 
 we find that $$ \#\{ m\in [0, M) : h \mid (a + m A) \} \leq \begin{cases} 
 0, & \text{if $(h, A) > 1$}, \\
 \frac{M}{h} + 1, & \text{if $(h, A) = 1$,} 
 \end{cases} $$
 so that \eqref{rewritehash} then gives us that 
\begin{align*}
 \sum_{m=0}^{M-1} d(a + m A) 
 & \leq 2 \sum_{\substack{ h \leq \sqrt{Y} \\ (h, A) = 1 }} 
 \left( \frac{M}{h} + 1 \right) \\ 
 & \leq 2 \sum_{\substack{ h \leq \sqrt{Y} }} 
 \left( \frac{M}{h} + 1 \right).
\end{align*}
 The desired result then follows by estimating this latter sum in a standard way. More explicitly, 
 we write
\begin{equation}\label{applyharmonic}
 2 \sum_{h \leq \sqrt{Y}} \left( \frac{M}{h} + 1 \right) 
 = 2 M \sum_{h \leq \sqrt{Y}} \frac{1}{h} + 2 \sum_{h \leq \sqrt{Y}} 1, 
\end{equation}
 and then apply to \eqref{applyharmonic} a standard harmonic number estimate. 
 Explicitly, this estimate is such that 
\begin{equation}\label{secondharmonic} 
 \sum_{h \leq \sqrt{Y}}\frac{1}{h} \leq 1 + \int_1^{\sqrt{Y}} \frac{dt}{t} = 1 + \frac{1}{2} \log Y,
\end{equation}
 with \eqref{applyharmonic} and \eqref{secondharmonic} together giving us that 
 $ \sum_{m=0}^{M-1} d(a + m A) $ $ \leq $ $ 2M \big( 1 + \frac{1}{2} \log Y \big) + 2 \sqrt{Y}$. 
 Moreover, if $\sqrt{Y} \leq M \log Y$, then (using the bound $Y \geq 3$) 
 we have that $2 M \big( 1 + \frac{1}{2} \log Y \big) \leq 3 M \log Y$ and $2 \sqrt{Y} \leq 2 M \log Y$, 
 and hence the desired bound in \eqref{withsqrtMlog}. 
\end{proof}

\begin{lemma}\label{taillemma}
 There exist absolute constants $C_{\operatorname{tail}} > 0$ and $X_0 \geq 3$ such that the following holds. For $X \geq X_0$, set 
 $\Lambda = \frac{\log X}{\log 2}$ and $k = \lfloor \Lambda^{1/10} \rfloor$
 and $L = \lfloor \Lambda^2 \rfloor$, and let $A, M \in \mathbb{N}$ and 	 $r \in \mathbb{N}_{0}$ and $Y \geq 3$, 
 and assume that 
\begin{align}
 Y & \leq 2^{L}, \label{Yleq2L} \\ 
 \sqrt{Y} & \leq M \log Y, \text{and} \label{sqrtYMlogY} \\ 
 r + (L-1) + (M-1) A & \leq Y. \label{longestassumption} 
\end{align}
 Also assume that: For a prime $p$, we have that 
\begin{equation}\label{ifpmidA}
 p \mid A \Longrightarrow p > L. 
\end{equation}
 Also assume that: For each prime divisor $p \mid A$, there is an integer $ j_p \in \{ 0, 1, \ldots, k-1 \} $ such that 
\begin{equation}\label{jpcongruence} 
 r + j_p \equiv 0 \pmod p. 
\end{equation}
 For 	 $m \in [0, M)$, set $ n_{m} = r + m A $ and $ T_{m, k} = \sum_{\ell \geq k} \frac{d(n_m + \ell)}{2^\ell}$. Then $$ \#\Big\{ m \in [0, 
 M) : T_{m, k} > 2^{-k/2} \Big\} \leq C_{\operatorname{tail}} M (\log Y) 2^{-k/2}.$$
\end{lemma}

\begin{proof}
 Choose $X_0$ to be large enough so as to guarantee that: For all $X \geq X_0$, we have that $k \geq 1$ and that $L \geq 2$ and that 
\begin{equation}\label{Lhalfk} 
 \frac{L}{2} \geq k 
\end{equation} 
and that 
\begin{equation}\label{sqrtLpower} 
 \sqrt{L} + 1 \leq 2^{L/2}.
\end{equation}
 This is possible, from the given definitions for $k$ and $L$. Rewrite $\sum_{m=0}^{M-1} T_{m, k}$ so that 
\begin{align*}
 \sum_{m=0}^{M-1} T_{m, k}
 & = \sum_{m=0}^{M-1} \sum_{\ell \geq k} \frac{d(r + m A + \ell)}{2^{\ell}} \\
 & = \sum_{\ell \geq k} 2^{-\ell} \sum_{m=0}^{M-1} d(r + \ell + m A). 
\end{align*}
 Writing 
\begin{align}
 S_1 & = \sum_{\ell = k}^{L-1} 2^{-\ell} \sum_{m=0}^{M-1} d(r + \ell + m A) \ \text{and} \label{displayS1} \\ 
 S_{2} & = \sum_{\ell \geq L} 2^{-\ell} \sum_{m=0}^{M-1} d(r + \ell + m A), \label{displayS2}
\end{align}
 i.e., with $ \sum_{m=0}^{M-1} T_{m, k} = S_{1} + S_{2}$. We proceed to bound $S_{1}$. 

 Fix 
\begin{equation}\label{ellkL}
 \ell \in [k, L), 
\end{equation}
 i.e., so that $\ell$ is among the indices in the outer sum in \eqref{displayS1}. We claim that 
\begin{equation}\label{coprimerellA}
 (r + \ell, A) = 1. 
\end{equation} 
 By way of contradiction, suppose that there exists a prime $p$ such that $p \mid A$ and $p \mid (r + \ell)$. Since $p \mid A$, this 
 \emph{same} prime is such that that there exists $j_p \in [0, k)$ such that \eqref{jpcongruence} holds. From \eqref{jpcongruence} 
 together with the congruence $ r + \ell \equiv 0 \pmod p$, we have, as a consequence, that 
\begin{equation}\label{elljp0p} 
 \ell - j_p \equiv 0 \pmod p. 
\end{equation}
 Since $\ell \geq k$ and since $j_p \leq k -1$, we have that $\ell - j_p \geq 1$. Since $\ell < L$ and since $j_p \geq 0$, we have that 
 $\ell - j_p < L$. Consequently, 
 the bounds $1 \leq \ell - j_p < L$ hold. By assumption, each prime divisor $p$ of $A$ is such that $p 
 > L$, i.e., so that $1 \leq \ell - j_p < p$, but this contradicts (via \eqref{elljp0p})
 that $p$ divides $\ell - j_p$. 
 This proves the relation in \eqref{coprimerellA}. 

 Now, set $a = r + \ell$. Recall that $r \in \mathbb{N}_{0}$. Also recall that $k \geq 1$, with $\ell \geq k$. So, the bound $a \geq 1$ 
 holds. As above, we have shown that $(a, A) = 1$ holds, i.e., so that the condition displayed in \eqref{aA1} within Lemma 
 \ref{basiclogbound} holds. We again have that $\ell \leq L-1$, recalling that we are working under the assumption that $k$ is in the 
 index set indicated in \eqref{ellkL}. With regard to the condition in \eqref{aM1Ay} associated with Lemma \ref{basiclogbound}, 
 we find that 
\begin{align*}
 a + (M-1) A & = r + \ell + (M-1) A \\ 
 & \leq r + (L-1) + (M-1) A, 
\end{align*}
 so that the assumption in \eqref{longestassumption} gives us that $a + (M - 1) A \leq Y$ holds, so that all of the fundamental 
 conditions for Lemma \ref{basiclogbound} are satisfied. Moreover, from the assumption in \eqref{sqrtYMlogY}, Lemma \ref{basiclogbound} 
 gives us (via \eqref{withsqrtMlog}) that $$ \sum_{m=0}^{M-1} d(r + \ell + m A) \leq 5 M \log Y, $$ i.e., so that $ S_{1} \leq 5 M \log Y 
 \sum_{\ell = k}^{L-1} 2^{-\ell}$, which implies that 
\begin{equation}\label{S1final} 
 S_{1} \leq 10 M (\log Y) 2^{-k}. 
\end{equation}

 We proceed to bound $S_2$. In this case, we exploit the elementary bound
\begin{equation}\label{dsqrt}
 d(N) \leq 2 \sqrt{N}, 
\end{equation}
 which follows from a pairing argument given previously. Starting with the assumption in \eqref{longestassumption}, from the above 
 assumption that $L \geq 2$, we have that $ r + (M-1) A \leq Y$, so that 
\begin{equation}\label{implyboundY}
 m \in [0, M) \Longrightarrow r + m A \leq Y. 
\end{equation}
 We henceforth let $\ell \geq L$, i.e., so that $\ell$ is an index associated with the outer sum in \eqref{displayS2}. Again letting $m \in [0, 
 M)$, from \eqref{implyboundY}, we write $r + \ell + m A \leq Y + \ell$, so that the divisor bound in \eqref{dsqrt} gives us that $ d(r + 
 \ell + m A) \leq 2 \sqrt{Y + \ell}$, i.e., so that 
\begin{equation}\label{S2elementary}
 S_{2} \leq 2 M \sum_{\ell \geq L} 2^{-\ell} \sqrt{Y + \ell}. 
\end{equation}
 Using the inequality $\sqrt{Y + \ell} \leq \sqrt{Y} + \sqrt{\ell}$, we have that 
\begin{equation}\label{detailsforS2}
 \sum_{\ell \geq L} 2^{-\ell} \sqrt{Y + \ell}
 \leq \sqrt{Y} \sum_{\ell \geq L} 2^{-\ell} + \sum_{\ell \geq L} 2^{-\ell} \sqrt{\ell}, 
\end{equation}
 by evaluating the first series on the right of \eqref{detailsforS2} in closed form, and by using the bound on $Y$ in \eqref{Yleq2L}, we 
 find that 
\begin{equation}\label{toreindex}
 \sum_{\ell \geq L} 2^{-\ell} \sqrt{Y + \ell}
 \leq 2^{1 - L/2} + \sum_{\ell \geq L} 2^{-\ell} \sqrt{\ell}. 
\end{equation}
 We then apply a reindexing argument to the right-hand series in \eqref{toreindex}, writing
\begin{equation}\label{secondsqrtineq}
 \sum_{\ell \geq L} 2^{-\ell} \sqrt{\ell} = 
 2^{-L} \sum_{t \geq 0} 2^{-t} \sqrt{L + t}, 
\end{equation}
 and we again exploit an elementary inequality of the form $\sqrt{L + t} \leq \sqrt{L} + \sqrt{t}$, with \eqref{secondsqrtineq} giving us 
 that $ \sum_{t \geq 0} 2^{-t} \sqrt{L + t} \leq \sqrt{L} \sum_{t \geq 0} 2^{-t} + \sum_{t \geq 0} 2^{-t} \sqrt{t} 
 \leq C_0 \big( \sqrt{L} + 1 \big)$, 
 for an absolute constant $C_0$. 
 Exploiting the bound 
$\sqrt{L} + 1 \leq 2^{L/2}$ in \eqref{sqrtLpower}, we obtain that 
\begin{equation}\label{aftersqrtLpow}
 \sum_{\ell \geq L} 2^{-\ell} \sqrt{\ell} 
 \leq C_0 2^{-L} \big( \sqrt{L} + 1 \big) \leq C_{0} 2^{-L/2}. 
\end{equation}
 From \eqref{toreindex} and \eqref{aftersqrtLpow}, we find that 
 $ \sum_{\ell \geq L} 2^{-\ell} \sqrt{Y + \ell} \leq 2^{1 - L/2} + C_{0} 2^{-L/2}$, i.e., so that $ \sum_{\ell \geq L} 2^{-\ell} \sqrt{Y + \ell} 
 \leq C_{1} 2^{-L/2} $ for an absolute constant $C_{1}$. Exploiting the bound in \eqref{Lhalfk}, we obtain that $ \sum_{\ell \geq L} 
 2^{-\ell} \sqrt{Y + \ell} \leq C_{1} 2^{-k}$, i.e., so that 
\begin{equation}\label{S2bound} 
 S_{2} \leq 2 C_{1} M 2^{-k}. 
\end{equation}
 Since $Y \geq 3$, we obtain from \eqref{S2bound} that 
\begin{equation}\label{S2final} 
 S_2 \leq 2 C_1 M (\log Y) 2^{-k}. 
\end{equation}
 From the bounds for $S_1$ and $S_2$ in \eqref{S1final} and \eqref{S2final}, we find that there is an absolute constant $C_{2}$ such that 
\begin{equation}\label{finalsumT} 
 \sum_{m=0}^{M-1} T_{m, k} \leq C_{2} M (\log Y) 2^{-k}. 
\end{equation}
 We proceed to set
 $ \mathcal{B} = \big\{ 0 \leq m < M : T_{m, k} > 2^{-k/2} \big\}$. 
 For $m \in \mathcal{B}$, we have that $T_{m, k} > 2^{-k/2}$, so that 
\begin{equation}\label{implicitMarkov} 
 \sum_{m \in \mathcal{B}} T_{m, k} \geq \sum_{m \in \mathcal{B}} 2^{-k/2} = | \mathcal{B} | 2^{-k/2}. 
\end{equation}
 From \eqref{implicitMarkov}, we proceed to write
\begin{equation}\label{afterMarkov}
 |\mathcal{B}| 2^{-k/2} \leq \sum_{m \in \mathcal{B}} T_{m, k} \leq \sum_{m=0}^{M-1} T_{m, k}, 
\end{equation}
 so that \eqref{finalsumT} and \eqref{afterMarkov} together give us the desired bound, 
 writing $C_{\operatorname{tail}} = C_{2}$. 
\end{proof}

\begin{theorem}
 The block $11$ appears infinitely often in the base-$2$ expansion of $E$. 
\end{theorem}

\begin{proof}
 Let $X$ be a real number (that is sufficiently large for the purposes of certain applications below). 
 Our strategy, at a basic level, is to construct
 an integer $n = n(X)$ such that 
 the $n^{\text{th}}$ and $(n+1)^{\text{th}}$ binary digits after the binary point of $E$ are both $1$. 
 We then let $X \to \infty$, and we formalize how this is permissible. 

 Set $\Lambda = \frac{\log X}{\log 2}$ and $k = \lfloor \Lambda^{1/10} \rfloor$ and $L = \lfloor \Lambda^{2} \rfloor$.
 For sufficiently large values of $X$, we have that $k \geq 3$. 

 We apply the formulation of the Alford--Granville--Pomerance lemma \cite{AlfordGranvillePomerance1994} (cf.\ \cite{Vandehey2013}) 
 given in Lemma \ref{Alfordetal}. Being consistent with the notation in Lemma \ref{Alfordetal}, we set $\delta = \frac{1}{4}$, and we set 
 $N = X$, i.e., with the assumption that $X$ is sufficiently large so that $X > N_0$. Again adopting the notation from 
 Lemma \ref{Alfordetal}, we write $\mathcal{D}(N) = \mathcal{D}(X)$
 to denote the exceptional set described in this lemma. So, by the AGP lemma, the inequality 
\begin{equation}\label{repeatAGP}
 \pi\big( X; d, a \big) \geq \frac{X}{2 \varphi(d) \log X} 
\end{equation}
 holds for all moduli $d \leq N^{1/4} = X^{1/4}$ not divisible by any exceptional element of $\mathcal{D}(N) = \mathcal{D}(X)$, 
 and \eqref{repeatAGP} also holds for $a$ coprime to $d$. 
 Our strategy, at this point, is to construct a modulus $B$ to which the 
 AGP lemma, with our specified values, is to be applied. 

 Define 
\begin{multline*}
 \mathcal{P}_{\operatorname{bad}}
 = \\ \Big\{ p \in (L, 2L) : \text{$p$ is prime and $p \mid D$ for some $D \in \mathcal{D}(X)$ with $D \leq X^{1/4}$} \Big\}. 
\end{multline*}
 We aim to estimate the size of $\mathcal{P}_{\operatorname{bad}}$. In this direction, 
 for a positive integer $D$, we set 
 $$ \omega_{(L, 2L)}(D) = \#\big\{ p \in (L, 2L) : \text{$p$ is prime and $p \mid D$} \big\}. $$
 If $\omega_{(L, 2L)}(D) = r$ for some integer $r$, then $D$ is divisible by the product
 of $r$ distinct primes, each exceeding $L$, i.e., so that 
 $ D \geq L^{r}$, so that 
 $ r \leq \frac{\log D}{\log L}$. 
 So, if $D \leq X^{1/4}$, then 
 $ \omega_{(L, 2L)}(D) \leq \frac{\log X}{4 \log L}$. 
 Now, by the 
 Alford--Granville--Pomerance lemma 
 (again adopting notation from Lemma \ref{Alfordetal}), 
 we have that the number of elements in $\mathcal{D}(N)$ is at most $\overline{\mathcal{D}} = \overline{\mathcal{D}}(\delta)$
 (depending only on $\delta = \frac{1}{4}$). We see that 
\begin{align*}
 |\mathcal{P}_{\operatorname{bad}}|
 & \leq \sum_{\substack{ D \in \mathcal{D}(X) \\ D \leq X^{1/4} }} \omega_{(L, 2L)}(D) \\ 
 & \leq \overline{\mathcal{D}} \frac{\log X}{4 \log L}. 
\end{align*}
 Since we have set $\delta = \frac{1}{4}$, and since $\overline{\mathcal{D}} = \overline{\mathcal{D}}(\delta)$ depends only on $\delta$, 
 we have that 
 $\overline{\mathcal{D}}$ is fixed, so that $ \left| \mathcal{P}_{\operatorname{bad}} \right| \ll \frac{\log X}{\log L}$. 

 Let 
\begin{equation}\label{mathcalAX}
 \mathcal{A}_{X} = \{ p \in (L, 2L) : \text{$p$ is prime and $p \not\in \mathcal{P}_{\operatorname{bad}}$} \}.
\end{equation}
 Our construction requires that $\mathcal{A}_{X}$ contain at least
\begin{equation}\label{closedsubstack} 
 1 + \sum_{\substack{ 0 \leq j < k \\ j \neq 2 }} (j+1)
 = \frac{k(k+1)}{2} - 2. 
\end{equation}
 For $k \geq 3$, the right-hand side of 
 \eqref{closedsubstack} is $\leq k^2$. 
 By the prime number theorem, there exists an absolute constant $c_{\pi} > 0$ such that:
 For all sufficiently large $L$, the bound
\begin{equation}\label{pifloor}
 \pi\left( \lfloor 2L \rfloor - 1 \right) - \pi(L) \geq c_{\pi} \frac{L}{\log L} 
\end{equation}
 holds. The definition of $\mathcal{A}_{X}$ 
 gives us, in a direct way, that the cardinality of \eqref{mathcalAX} satisfies 
\begin{equation}\label{cardinalityrel} 
 \left| \mathcal{A}_{X} \right| = 
 \pi\left( \lfloor 2L \rfloor - 1 \right) - \pi(L) - \left| \mathcal{P}_{\operatorname{bad}} \right|. 
\end{equation}
 Applying \eqref{pifloor} and the upper bound for $| \mathcal{P}_{\operatorname{bad}} |$ established
 above to the cardinality relation in \eqref{cardinalityrel}, we find that 
\begin{equation}\label{boundcardinal} 
 \left| \mathcal{A}_{X} \right| \geq c_{\pi} \frac{L}{\log L} - \overline{\mathcal{D}} \frac{\log X}{4 \log L}. 
\end{equation}
 Recall that $\Lambda = \frac{\log X}{\log 2}$ and that $L = \lfloor \Lambda^{2} \rfloor$. 
 So, for all sufficiently large $X$, we have that 
\begin{equation}\label{Lquadratically} 
 L \geq \frac{1}{2} \Lambda^{2} = \frac{ (\log X)^{2} }{2 (\log 2)^2}. 
\end{equation}
 Since $L$ grows quadratically in $\log X$ according to \eqref{Lquadratically}, we can conclude
 that: For sufficiently large $X$, the relation 
\begin{equation}\label{dividelog} 
 \overline{\mathcal{D}} \frac{\log X}{4} \leq \frac{c_{\pi}}{2} L 
\end{equation}
 holds. From \eqref{boundcardinal} and \eqref{dividelog} together, we may deduce that 
\begin{equation}\label{boundAlikepi}
 \left| \mathcal{A}_{X} \right| \geq \frac{c_{\pi}}{2} \frac{L}{\log L}. 
\end{equation}
 Since $k = \lfloor \Lambda^{1/10} \rfloor$, we have (for sufficiently large $X$) that 
 $ k^2 \leq \Lambda^{1/5}$. 
 In a similar spirit, since 
 $L = \lfloor \Lambda^{2} \rfloor$, we have that 
 $ \frac{L}{\log L} \asymp \frac{\Lambda^2}{\log \Lambda}$. 
 Consequently, we obtain the relations
 $ \frac{L/\log L}{k^2} \gg \frac{ \Lambda^2 / \log \Lambda }{\Lambda^{1/5}} = 
 \frac{\Lambda^{9/5}}{\log \Lambda} \to \infty $
 as $X \to \infty$. 
 So, for sufficiently large $X$, the relation 
\begin{equation}\label{lowerksquare}
 \frac{c_{\pi}}{2} \frac{L}{\log L} \geq k^2 
\end{equation}
 holds. From \eqref{boundAlikepi} and \eqref{lowerksquare}, since $ |\mathcal{A}_{X}| \geq k^2$, 
 we can deduce that: For sufficiently 
 large $X$, the number of primes in the set $(L, 2L) \setminus \mathcal{P}_{\operatorname{bad}}$
 is greater than or equal to the left-hand side of \eqref{closedsubstack}. 

 From the argument in the preceding paragraph, we are permitted, for sufficiently large $X$, 
 to choose distinct primes $q_0$
 and $p_{j, t}$ for $0 \leq j < k$ and $j \neq 2$ and $1 \leq t \leq j +1$
 from the set $(L, 2L) \setminus \mathcal{P}_{\operatorname{bad}}$. 
 For $j \neq 2$, define
\begin{equation}\label{Pprodprime}
 P_j = \prod_{t=1}^{j+1} p_{j, t}. 
\end{equation}
 Also define
\begin{equation}\label{Aprodprime} 
 A = q_0^3 \prod_{\substack{ 0 \leq j < k \\ j \neq 2 }} P_{j}^2 
\end{equation}
 and 
\begin{equation}\label{Bprodprime}
 B = \frac{A}{q_0^2} = q_0 \prod_{\substack{ 0 \leq j < k \\ j \neq 2 }} P_j^2. 
\end{equation}

 We claim that 
\begin{equation}\label{Broot} 
 B \leq X^{1/4}
\end{equation}
 for sufficiently large $X$. Each chosen prime is $ < 2 L$, 
 and the total number of prime factors of $B$, counted with multiplicity, is $O(k^2)$. So, we find that 
 $ \log B = O(k^2 \log L) = O( \Lambda^{1/5} \log \Lambda ) = o(\log X)$, 
 and hence \eqref{Broot} holding for sufficiently large $X$. 

 By way of contradiction, suppose that there exists
 an element $D \in \mathcal{D}(X)$ that divides $B$. From \eqref{Broot}, we find that 
\begin{equation}\label{DBXquarter} 
 D \leq B \leq X^{1/4} 
\end{equation}
 for all sufficiently large $X$. Since each member of $\mathcal{D}(X)$ exceeds $\log X$, 
 we have, for sufficiently large $X$, that $D > 1$. 
 So, the integer $D$ has at least one prime divisor. From the assumption that $D \mid B$, 
 every prime that divides $D$ also divides $B$, and thus is among the chosen primes from $(L, 2L)$. 
 We thus choose a prime divisor $p$ of $D$, but, since $D \in \mathcal{D}(X)$
 and since \eqref{DBXquarter} gives us that $D \leq X^{1/4}$, 
 we have that $p \in \mathcal{P}_{\operatorname{bad}}$. 
 This contradicts that the chosen primes were chosen outside of $\mathcal{P}_{\operatorname{bad}}$.
 So, we have that no element in $\mathcal{D}(X)$ divides $B$. 

 Now, recall the formulation of the CRT given above. 
 Our construction gives us, in a direct way, 
 that expressions of the forms $q_0^3$ and $P_j^2$ for $0 \leq j < k$ and $j \neq 2$
 are pairwise coprime. 
 So, the CRT gives us that there is a residue $r$ simultaneously satisfying 
\begin{equation}\label{firstequivCRT} 
 r \equiv q_0^2 - 2 \pmod {q_{0}^{3}} 
\end{equation}
 and 
\begin{equation}\label{secondequivCRT} 
 r \equiv P_j - j \pmod {P_{j}^{2}} 
\end{equation}
 for $0 \leq j < k$ and $ j \neq 2$. Moreover, 
 the above residue $r$ is unique modulo the product of the moduli, which is precisely $A$ as defined in 
 \eqref{Aprodprime}. We then fix $r$ as the unique residue class such that 
\begin{equation}\label{rbounds} 
 0 \leq r < A.
\end{equation}
 From \eqref{firstequivCRT}, we find that 
 $r+2$ is divisible by $q_0^2$, and this leads us to define 
\begin{equation}\label{defines} 
 s = \frac{r+2}{q_0^2}. 
\end{equation}
 Since $0 \leq r < A = q_0^2 B$, we deduce that 
\begin{equation}\label{rplus2bound} 
 0 < r + 2 \leq q_0^2 B + 1. 
\end{equation}
 Since $q_0^2 \mid (r+2)$, this and the definition of $s$ in \eqref{defines}
 and the inequalities in \eqref{rplus2bound} gives us that 
\begin{equation}\label{sbounds} 
 1 \leq s \leq B. 
\end{equation}
 From \eqref{firstequivCRT} and the definition of $s$, 
 we see that 
 $ q_0^2 s \equiv q_0^2 \pmod {q_{0}^{3}}$. 
 Since $q_0^3 \mid q_0^2 (s-1)$, 
 it follows that $q_0 \mid (s-1)$, 
 i.e., so that 
\begin{equation}\label{s1modq} 
 s \equiv 1 \pmod {{q}_{0}}.
\end{equation}
 Recalling \eqref{Bprodprime}, we see that $q_0 \mid B$. From this divisibility relation
 and the congruence in \eqref{s1modq}, we deduce that it is not the case that $s= B$. 
 We thus may refine \eqref{sbounds}, writing 
\begin{equation}\label{stricts} 
 1 \leq s < B. 
\end{equation} 

 We claim that 
\begin{equation}\label{sB1} 
 (s, B) = 1. 
\end{equation}
 To begin with, the relation in \eqref{s1modq}
 allows us to deduce that $q_0 \nmid s$. 
 Now, write $p$ in place of one of the other prime divisors of $B$. 
 Observe that $p \mid P_j$, i.e., for some index $j \in [0, k) \setminus \{ 2 \}$
 associated with the product in \eqref{Bprodprime} (recalling the condition such that $j \neq 2$). 
 Moreover, the congruence in 
 \eqref{secondequivCRT} allows us to write 
 $r + j = P_j^2 z + P_j $ for some integer $z$, 
 so that 
 $ r + j \equiv 0 \bmod p$, 
 so that 
 $ r + 2 \equiv 2 - j \bmod p$, 
 i.e., so that 
\begin{equation}\label{q0squareds} 
 q_0^2 s \equiv 2 - j \pmod p. 
\end{equation}
 By way of contradiction, suppose that $p \mid s$. With this assumption, 
 the equivalence in 
 \eqref{q0squareds} allows us to deduce that 
\begin{equation}\label{2minusjequiv}
 2 - j \equiv 0 \pmod p. 
\end{equation}
 Since $j \neq 2$ and $2 \geq 2 - j > 2 -k$ and $p > L$ (recalling that $p \in (L, 2L)$ by construction)
 and $L > k$ (recalling that $k = \lfloor \Lambda^{1/10} \rfloor$ and $L = \lfloor \Lambda^2 \rfloor$), 
 we find that $0 < |2 - j | < p$, contradicting (via \eqref{2minusjequiv})
 that $p \mid (2-j)$. So, the relation $p \nmid s$ holds, so that 
 \eqref{sB1} holds, as desired. 

 We proceed to apply the AGP estimate in Lemma \ref{Alfordetal} and in \eqref{repeatAGP}, with $d = B$ and $a = s$. We are permitted 
 to apply Lemma \ref{Alfordetal}, since: 

\vspace{0.1in}

\noindent $\bullet$ $B \leq X^{1/4}$ by \eqref{Broot}; 

\vspace{0.1in}

\noindent $\bullet$ $B$ is not divisible by any element of $\mathcal{D}(X)$, from the preceding paragraph; and 

\vspace{0.1in}

\noindent $\bullet$ $s$ and $B$ are coprime, as established in the preceding paragraph. 

\vspace{0.1in}

\noindent So, since the required conditions of Lemma \ref{Alfordetal} are satisfied, we find that 
\begin{equation}\label{applyAlford}
 \pi(X; B, s) \geq \frac{X}{ 2 \varphi(B) \log X }, 
\end{equation}
 recalling that we are writing $X$ in place of $N$. Since $\varphi(B) \leq B$, we obtain 
 from \eqref{applyAlford} that 
\begin{equation}\label{piX2BlogX} 
 \pi(X; B, s) \geq \frac{X}{ 2 B \log X }. 
\end{equation}

 Now, define 
\begin{equation}\label{definecapM}
 M_{X, B} = M = \left\lfloor \frac{X}{B} \right\rfloor + 1. 
\end{equation}
 For $M$ as defined in \eqref{definecapM}, and for arbitrary $m \in [0, M)$, 
 we then define 
\begin{equation}\label{nmwithoutp} 
 n_m = r + m A. 
\end{equation}
 Again for $M$ as in \eqref{definecapM}, we also define 
\begin{equation}\label{mathcalG}
 \mathcal{G}_{X, M} = \mathcal{G}_{X} = \{ 0 \leq m < M : s + m B \leq X, \ s + m B \ \text{prime} \}. 
\end{equation}

 Since $1 \leq s < B$ (recalling \eqref{stricts}), the map $m \mapsto s + m B$ provides a bijection from $\mathcal{G}_{X}$ onto the set of 
 primes $p \leq X$ satisfying $p \equiv s \bmod B$. Indeed,
 if $p \leq X $ and $p \equiv s \bmod B$, then $p = s + m B$ for a unique integer $m$.
 Since $p > 0$ and since $1 \leq s < B$ (again recalling \eqref{stricts}),
 the $m < 0$ case would be impossible, because, otherwise, 
 we would have that $p = s + m B \leq s - B < 0$. 
 So, we obtain the nonnegativity of $m \geq 0$.
 Since $p \leq X$ and $s \geq 1$, we find that
 $m B = p - s < X$, i.e., so that 
 $0 \leq m < \frac{X}{B} < M$, with $m \in \mathcal{G}_{X}$. 
 Conversely, each element $m \in \mathcal{G}_{X}$ gives rise to a prime of the form $p = s + m B$, 
 and hence the equality $|\mathcal{G}_{X}| = \pi(X;B,s)$. 
 
 Recalling (via \eqref{Broot}) that $B \leq X^{1/4}$, we find that $\frac{X}{B} \to \infty$, and we can deduce that: For sufficiently large $X$, 
 the relation $ M \leq \frac{2X}{B} $ holds. From \eqref{piX2BlogX}, we have, as a consequence, that 
\begin{equation}\label{Mover4logX}
 |\mathcal{G}_{X}| = \pi(X; B, s) \geq \frac{M}{4 \log X} 
\end{equation}
 holds for sufficiently large $X$. 
 
 For $m \in \mathcal{G}_{X}$, set $ p = s + m B$. Then $p$ is a prime and $p \leq X$. Moreover, the definition in \eqref{nmwithoutp} 
 gives us that 
\begin{equation}\label{nmplus2} 
 n_m + 2 = r + 2 + m A. 
\end{equation}
 The definition of $s$ in \eqref{defines} gives us that 
 $ r + 2 = q_0^2 s$, 
 and the definition of $B$ in \eqref{Bprodprime}
 gives us that $ A = q_0^2 B$. 
 So, the equality in \eqref{nmplus2} then gives us that 
\begin{equation}\label{nm2factor} 
 n_m + 2 = q_0^2 s + m q_0^2 B = q_0^2 (s + m B) = q_0^2 p. 
\end{equation}
 From \eqref{Bprodprime}, we obtain the divisibility relation 
 $q_0 \mid B$, so that 
 $ p = s + m B \equiv s \bmod {{q}_{0}}$. 
 From the congruence in \eqref{s1modq}, we then find that 
\begin{equation}\label{pequiv1modq} 
 p \equiv 1 \pmod {{q}_{0}}, 
\end{equation}
 and we can see from \eqref{pequiv1modq}
 that $p \neq q_0$. 
 So, from \eqref{nm2factor}, we see that 
\begin{equation}\label{maind6} 
 d(n_m +2) = d(q_0^2 p) = 6. 
\end{equation}
 So, a combined application of \eqref{Mover4logX} and \eqref{maind6} gives us that 
\begin{equation}\label{GpiMX} 
 \#\big\{ m \in [0, M) : d(n_m + 2) = 6 \big\} 
 \geq |\mathcal{G}_{X}| =\pi(X;B,s) \geq \frac{M}{4 \log X}. 
\end{equation}

 Now, we proceed to apply the tail estimate given in Lemma \ref{taillemma}, letting 
\begin{equation}\label{Yfortail}
 Y = 2 q_0^2 X.
\end{equation}
 We demonstrate, as below, that the required conditions
 in Lemma \ref{taillemma} are all satisfied. 

 Recall that $q_0$ is chosen from the set $ (L, 2L) \setminus \mathcal{P}_{\operatorname{bad}}$, 
 i.e., so that $q_0 < 2 L$. This together with \eqref{Yfortail}
 give us that 
\begin{equation}\label{applylog2}
 Y \leq 8 L^2 X. 
\end{equation} 
 We let $\log_2$ denote the base-2 logarithm. Applying this to both sides of 
 \eqref{applylog2}, and using the definition of $\Lambda = \log_2 X$, 
 we find that 
\begin{equation}\label{boundlog2Y}
 \log_2 Y \leq \Lambda + O(\log L). 
\end{equation}
 Since $L = \lfloor \Lambda^2 \rfloor$, applying this in conjunction with \eqref{boundlog2Y} allows us to deduce that: For sufficiently 
 large $X$, the bound $ \log_2 Y \leq L $ holds, i.e., so that $Y \leq 2^L$, giving us that the condition in \eqref{Yleq2L} holds for 
 sufficiently large $X$. 

 Recall (from \eqref{Broot}) that $B \leq X^{1/4}$. This gives us that 
\begin{equation}\label{MXthreequarters}
 M = \left\lfloor \frac{X}{B} \right\rfloor + 1 \geq \frac{X}{B} \geq X^{3/4}. 
\end{equation}
 From the definition in \eqref{Yfortail} together with the bound $q_0 < 2 L$, we find that 
 $ \sqrt{Y} \leq 2 \sqrt{2} L X^{1/2}$. 
 Since $Y = 2 q_0^2 X \geq X$, we find that $\log Y \geq \log X$. So, we obtain the lower bound
\begin{equation}\label{MlogYoversqrtY}
 \frac{M \log Y}{\sqrt{Y}} \geq \frac{X^{3/4} \log X}{2 \sqrt{2} L X^{1/2}} = \frac{1}{2\sqrt{2}} \frac{X^{1/4} \log X}{L}. 
\end{equation}
 Since $L = \lfloor \Lambda^{2} \rfloor$ and $\Lambda = \frac{\log X}{\log 2}$, we find that 
\begin{equation}\label{reciprocalL} 
 \frac{1}{L} \geq \frac{ (\log 2)^{2} }{ (\log X)^2 }. 
\end{equation}
 As a consequence of both \eqref{MlogYoversqrtY} and \eqref{reciprocalL}, we find that 
\begin{equation}\label{betterMlogY} 
 \frac{M \log Y}{\sqrt{Y}} \geq \frac{\log^2 2}{2\sqrt{2}} \frac{X^{1/4}}{\log X}. 
\end{equation}
 Since the right-hand side of \eqref{betterMlogY} tends to infinity as $X \to \infty$, we can conclude that $\frac{M \log Y}{\sqrt{Y}} \geq 
 1$ for all sufficiently large $X$, i.e., so that $\sqrt{Y} \leq M \log Y$ for all sufficiently large $X$, i.e., 
 so that the desired condition in \eqref{sqrtYMlogY} within Lemma \ref{taillemma} holds. 

 Recalling the bounds on $r$ in \eqref{rbounds}, and recalling the relation $A = q_0^2 B$  given in \eqref{Bprodprime},  we find that 
\begin{equation}\label{reboundr} 
 0 \leq r < q_0^2 B.
\end{equation}
 Also, recalling the definition of $M$ on display in \eqref{definecapM}, 
 we find that 
\begin{equation}\label{boundMminus1}
 M - 1 \leq \frac{X}{B}. 
\end{equation}
 A combined application of \eqref{reboundr} and \eqref{boundMminus1} then gives us that 
\begin{equation}\label{withouto} 
 r + (L-1) + (M-1) A \leq q_0^2 B + L + q_0^2 X. 
\end{equation}
 Recalling, from \eqref{Broot}, that $B \leq X^{1/4}$, and recalling that $L \asymp (\log X)^{2}$
 (since $L = \lfloor \Lambda^{2} \rfloor$ and $\Lambda = \frac{\log X}{\log 2}$), we have that 
\begin{equation}\label{witho} 
 q_0^2 B + L = o\big( q_0^2 X \big). 
\end{equation}
 A combined application of \eqref{Yfortail}, \eqref{withouto}, and \eqref{witho} then 
  gives us that: For sufficiently large $X$, the bound
 $ r + (L-1) + (M-1) A \leq 2 q_0^2 X = Y $
 holds. So, the desired condition in \eqref{longestassumption} within 
 Lemma \ref{taillemma} holds. 
 
 Now, every prime divisor of $A$ is a prime chosen in $(L, 2L)$, i.e., so that if $p \mid A$, then $p > L$. So, the desired implication 
 in \eqref{ifpmidA} holds. 

 Now, for a prime divisor $p \mid A$, we want to determine a value
 $ j_p \in \{ 0, 1, \ldots, k - 1 \}$, 
 such that 
\begin{equation}\label{rjp0p} 
 r + j_p \equiv 0 \pmod p, 
\end{equation}
 noting that we are working under the assumption that $k \geq 3$. Recall that $q_0^2 \mid (r+2)$, as established above, in our 
 construction of the value $s$ defined in \eqref{defines}. So, we find that the prime $q_0$ divides $r+2$. So, if $p$ happens to be equal 
 to $q_0$, then, in this case, we set $j_p = j_{q_0} = 2$, i.e., so that the desired congruence relation in \eqref{rjp0p} holds. Now, recall the 
 consequence of the CRT in \eqref{secondequivCRT}. We proceed to rewrite the congruence in \eqref{secondequivCRT} so that $ r + j 
 \equiv P_j \pmod {P_{j}^{2}}$. Since $p_{j, t}$ divides $P_j$ (and $P_j^2$), and since ${P_{j}^{2}}$ divides $ r + j - P_j$, we have that 
\begin{equation}\label{pjtmid} 
 p_{j, t} \mid (r + j - P_j).
\end{equation} 
 Again since $p_{j, t}$ divides $P_j$, we deduce from \eqref{pjtmid} that $p_{j, t}$ divides $r + j$. So, if $p = p_{j, t}$, then we let $j_p = 
 j$, so that \eqref{rjp0p} is satisfied. 
 So, we have that the final condition of Lemma \ref{taillemma} holds, in reference
 to the condition involving \eqref{jpcongruence}. 

 So, from our above construction, since all of the conditions of Lemma \ref{taillemma} are satisfied, we obtain from Lemma 
 \ref{taillemma} that: There exists an absolute constant $C_{\operatorname{tail}} > 0$ such that: For all sufficiently large $X$, 
 the inequality 
\begin{equation}\label{bigenumerate} 
 \#\left\{ m \in [0, M) : \sum_{\ell \geq k} \frac{d(n_m + \ell)}{2^\ell} > 2^{-k/2} \right\} 
 \leq C_{\operatorname{tail}} M (\log Y) 2^{-k/2} 
\end{equation}
 holds. 

 From the definition of $Y$ in \eqref{Yfortail}, we see that $ \log Y = O(\log X)$. Moreover, since $k = \lfloor \Lambda^{1/10} \rfloor$, 
 we have that $\Lambda^{1/10} - 1 < k \leq \Lambda^{1/10}$, so that $\frac{1}{2} \Lambda^{1/10} - \frac{1}{2} < \frac{k}{2} \leq 
 \frac{1}{2} \Lambda^{1/10}$, so that $ \frac{k}{2} = \frac{1}{2} \Lambda^{1/10} + O(1)$, so that, for any positive power $\alpha$, 
 we obtain  that 
\begin{align*}
 2^{k/2} & = 2^{\frac{1}{2} \Lambda^{1/10} + O(1)} \\ 
 & = \exp\left( \left( \log 2 \right) \left( \frac{1}{2} \Lambda^{1/10} + O(1) \right) \right) \\ 
 & = \exp\left( \frac{\log 2}{2} \Lambda^{1/10} + O(1) \right) \\ 
 & \gg \left( \log X \right)^{\alpha}. 
\end{align*}
 So, since $2^{-k/2} \ll \frac{1}{(\log X)^{\alpha}}$, we find (setting $\alpha =3$) that $ C_{\operatorname{tail}} M (\log Y) 2^{-k/2} = 
 o\left( \frac{M}{\log X} \right)$. So, for all sufficiently large $X$, the inequality 
\begin{equation}\label{upperM4logX} 
 C_{\operatorname{tail}} M (\log Y) 2^{-k/2} < \frac{M}{4 \log X} 
\end{equation}
 holds. From \eqref{GpiMX}, 
 there are \emph{at least} $\frac{M}{4 \log X}$ integers $m \in [0, M)$
 such that $d(n_m + 2) = 6$, 
 From \eqref{bigenumerate} and \eqref{upperM4logX}, 
 there are \emph{strictly fewer} than $\frac{M}{4 \log X}$ integers $m \in [0, M)$ 
 satisfying 
 $ \sum_{\ell \geq k} \frac{d(n_m + \ell)}{2^\ell} > 2^{-k/2}$. 
 So, using a ``pigeonhole-like'' counting argument, there is \emph{at least one} integer $m \in [0, M)$ such that both 
\begin{equation}\label{clarityrewrite1}
 d(n_m + 2) = 6 
\end{equation}
 and 
\begin{equation}\label{clarityrewrite2} 
 \sum_{\ell \geq k} \frac{d(n_m + \ell)}{2^\ell} \leq 2^{-k/2} 
\end{equation}
 hold. We fix such an integer $m$, and write $n = n_m$. Writing $n = n_m$, the relations in \eqref{clarityrewrite1} and 
 \eqref{clarityrewrite2} become 
\begin{equation}\label{dnplus2} 
 d(n + 2) = 6. 
\end{equation}
 and 
\begin{equation}\label{simplifytail} 
 \sum_{\ell \geq k} \frac{d(n + \ell)}{2^{\ell}} \leq 2^{-k/2}. 
\end{equation}

 Let $j \neq 2$, with $j \in [0, k)$. Recalling the definition in \eqref{Aprodprime}, we have that $P_j^2$ divides $A$. So, since $n+j = r + 
 m A +j$, we have that 
\begin{equation}\label{nplusjequiv} 
 n + j \equiv (r+j) \pmod {P_j^2}. 
\end{equation}
 Our application of the CRT in \eqref{secondequivCRT} then gives us, in conjunction with 
 \eqref{nplusjequiv}, that 
\begin{equation}\label{Pjmodsquare} 
 n + j \equiv P_j \pmod {P_{j}^{2}}. 
\end{equation} 
 Recalling the definition in \eqref{Pprodprime}, we find that: For $t = 1, 2, \ldots, j+1$, the prime $p_{j, t}$ divides $P_j$, but $p_{j, t}^2 
 \nmid P_j$. From \eqref{Pjmodsquare}, we have that 
\begin{equation}\label{capPjsquaredmid}
 P_j^2 \mid (n + j - P_j). 
\end{equation}
 So, since $p_{j, t} \mid P_j^2$ and $p_{j, t} \mid P_j$, we have that $p_{j, t} \mid (n + j)$. We also claim that $p_{j, t}^{2} \nmid (n+j)$, 
 and we proceed by contradiction,  assuming instead that 
\begin{equation}\label{contradictpjt} 
 p_{j, t}^{2} \mid (n+j).
\end{equation} 
 From \eqref{capPjsquaredmid}, again since $p_{j, t}^2 \mid P_{j}^2$, we have that 
\begin{equation}\label{usetocontradict} 
 p_{j, t}^2 \mid (n + j - P_j). 
\end{equation}
 From \eqref{contradictpjt} and \eqref{usetocontradict}, these together would imply that $p_{j, t}^{2} \mid P_j$, which does not hold. 
 So, we have shown that $p_{j, t} \mid (n + j)$ and that $p_{j, t}^{2} \nmid (n + j)$ for fixed $j$ and for the indices $t \in [1, j+1]$ 
 associated with the product in \eqref{Pprodprime}. So, the integer $n+j$ has at least $j+1$ distinct prime divisors that each have an 
 exponent exactly equal to $1$ in the prime factorization of $n + j$. Therefore, the relation 
\begin{equation}\label{2jplus1mid} 
 2^{j+1} \mid d(n+j) 
\end{equation}
 holds for all $j \in [0, k) \setminus \{2 \}$.

 Now, we rewrite the latter series expansion for the Erd\H{o}s--Borwein constant in \eqref{displayE} so that 
\begin{equation}\label{rewriteE}
 E = \sum_{1 \leq t < n} \frac{d(t)}{2^t} + \sum_{\substack{0\leq j < k \\ j \neq 2}} \frac{d(n+j)}{2^{n+j}} + \frac{d(n+2)}{2^{n + 
 2}} + \sum_{\ell \geq k} \frac{d(n +\ell)}{2^{n + \ell}}. 
\end{equation}

 Now, write 
\begin{equation}\label{mathcalP} 
 \mathcal{P} = \sum_{1 \leq t < n} \frac{d(t)}{2^t} 
 + \sum_{\substack{0 \leq j < k \\ j \neq 2}} \frac{d(n+j)}{2^{n+j}}. 
\end{equation}
 Our goal, at this point, is to show that $\mathcal{P}$ is an integer multiple of $2^{-(n-1)}$. In this direction, if $1 \leq t < n$, then 
 $ \frac{d(t)}{2^t} = d(t) 2^{n - 1 - t} \, 2^{-(n-1)} $ is an integer multiple of $2^{-(n-1)}$. If $j \in [0, k) \setminus \{2 \}$, then, as above, 
 the relation in \eqref{2jplus1mid} holds, so that $ \frac{d(n+j)}{2^{n+j}} = \frac{d(n+j)/2^{j+1} }{2^{n-1}} $ is an integer multiple of 
 $2^{-(n-1)}$. So, there is an integer $\mathcal{Q}$ satisfying
\begin{equation}\label{PtoQ} 
 \mathcal{P} = \mathcal{Q} 2^{-(n-1)}. 
\end{equation} 
 From \eqref{dnplus2}, write
\begin{equation}\label{isolatedcase} 
 \frac{d(n+2)}{2^{n+2}} = \frac{3}{4} 2^{-(n-1)}. 
\end{equation}
 By analogy with \eqref{mathcalP}, we write 
\begin{equation}\label{mathcalR}
 \mathcal{R} = \sum_{\ell \geq k} \frac{d(n + \ell)}{2^{n + \ell}}. 
\end{equation}
 Now, the simplified tail estimate in \eqref{simplifytail} gives us that $ \mathcal{R} = 2^{-n} \sum_{\ell \geq k} \frac{d(n + \ell)}{2^{\ell}} 
 \leq 2^{-n} 2^{-k/2} = 2^{-n-k/2}$. Since we are and have been consistently working under the assumption that $k \geq 3$, we 
 find that 
\begin{equation}\label{boundmathcalR}
 \mathcal{R} < 2^{-n-1} = \frac{1}{4} 2^{-(n-1)}. 
\end{equation}
 Now, we rewrite the right-hand side of \eqref{rewriteE} through a combined use of \eqref{PtoQ}, \eqref{isolatedcase}, and 
 \eqref{mathcalR}, with $ E = \mathcal{Q} 2^{-(n-1)} + \frac{3}{4} 2^{-(n-1)} + \mathcal{R}$, where \eqref{boundmathcalR} gives us that 
 $ 0 \leq \mathcal{R} < \frac{1}{4} 2^{-(n-1)}$. So, we find that $ 2^{n-1} E = \mathcal{Q} + \frac{3}{4} + \theta $ for some real number 
 $\theta \in \big[0, \frac{1}{4} \big)$. So, the fractional part of $2^{n-1} E$ is in 
 $\big[ \frac{3}{4}, 1 \big)$. Since $E$ is irrational \cite{Erdos1948}, the binary expansion of $E$ is unique 
 (which avoids any ambiguity regarding the possibility of non-terminating binary decimal expansions for rational values). 
 Consequently, the first two binary digits after the binary point of $2^{n-1}E$ are both $1$.
 Consequently, the $n^{\text{th}}$ and $(n+1)^{\text{th}}$ binary digits after the binary point of $E$ are both $1$. 

 Now, it remains to show (as below) that the $n^{\text{th}}$ and $(n+1)^{\text{th}}$ positions above
 are unbounded as $X \to \infty$. 

 Since $n = r + m A$ and $A$ is divisible by $q_0^3$, while $r + 2 \equiv q_0^2 \bmod q_0^3$, we have that 
\begin{equation}\label{usefornu}
 n + 2 \equiv q_0^2 \bmod q_0^3. 
\end{equation}
 From the congruence in \eqref{usefornu}, we see that $n + 2 - q_0^2 = z q_0^3 $ for some integer $z$, so that $ n + 2 = q_0^2(1 + 
 z q_0)$. This together with the inequivalence $1 + z q_0 \not\equiv 0 \bmod q_0$ allow us to conclude that $\nu_{q_0}(n+2)=2$ holds. 
 Also, our construction gives us that $d(n + 2) =6$. If $z$ is a positive integer with $d(z) = 6$, then either $z = p^5$ for some prime $p$ 
 or $z = p^2 q$ for distinct primes $p$ and $q$. Since $\nu_{q_{0}}(n+2)=2$, the first case is impossible, so that the second case forces that 
 $n + 2 = q_0^2 p$ for some prime $p \neq q_0$. Consequently, the inequality $n + 2 \geq 2 q_0^2$ holds. Since $q_0 > L = \lfloor 
 \Lambda^2 \rfloor \to \infty$ as $X \to \infty$, we find that $n \to \infty$. This proves that the positions at which $11$ occurs 
 are unbounded. 
\end{proof}

\section{Conclusion}
 As a natural follow-up to our construction, one might consider the problem of determining whether or not \emph{every} binary string 
 appears infinitely often in the base-$2$ expansion of $E$. We greatly encourage the exploration of this. 

\subsection*{Acknowledgements}
The author acknowledges extensive interactions with 
 GPT-5.5 Pro during the exploratory and proof-development stages of this work. All AI-generated suggestions were substantially revised, corrected, and independently verified by the author, who assumes full responsibility for the mathematical content.

\end{document}